\documentclass[preprint,12pt]{elsarticle}

\usepackage{dcolumn}
\usepackage{bm}
\usepackage{amssymb}
\usepackage{subfigure}
\DeclareMathAlphabet{\mathpzc}{OT1}{pzc}{m}{it}
\usepackage{amsthm}
\newtheorem{theorem}{Theorem}
\newtheorem{corollary}{Corollary}
\newtheorem{remark}{Remark}
\newtheorem{example}{Example}

\newcommand{\be}{\begin{equation}}
\newcommand{\ee}{\end{equation}}

\begin{document}
\begin{frontmatter}
\title{A note on mean-value properties of harmonic functions on the hypercube}

\author[FMI]{P. P.~Petrov\corref{cor1}}%
\address[FMI]{Faculty of Mathematics and Informatics, Sofia University, 5 James Bourchier blvd., 1164 Sofia, Bulgaria}
\ead{peynov@fmi.uni-sofia.bg}
\cortext[cor1]{Corresponding author; Phone: +359\,2\,8161\,506, Fax: +359\,2\,868\,71.} 

\begin{abstract}
For functions defined on the $n$-dimensional hypercube $I_n (r) = \{{\bm{x}} \in \mathbb{R}^n ~\vert~ \vert x_i \vert \le r,~ i = 1, 2, \ldots , n\}$ and harmonic therein, we establish certain analogues of Gauss surface and volume mean-value formulas for harmonic functions on  the ball in $\mathbb{R}^n$ and their extensions for polyharmonic functions. The relation of these formulas to best one-sided  $L^1$-approximation by harmonic functions on $I_n (r)$ is also discussed.
\end{abstract}
%

\end{frontmatter}


\section{Introduction}

This note is devoted to formulas for calculation of integrals over the $n$-dimensional hypercube centered at $\bm{0}$
$$
 I_n:=I_n (r) := \{\bm{x} \in \mathbb{R}^n ~\vert~ \vert x_i \vert \le r,~ i = 1,2, \ldots , n\},~r>0,
$$
and its boundary $P_n:=P_n(r):=\partial I_n(r)$, based on integration over hyperplanar subsets of $I_n$ and  exact for harmonic or polyharmonic functions. 
They are presented in Section~\ref{sec-mean} and can be considered as natural analogues on $I_n$ of Gauss surface and volume mean-value formulas for harmonic functions (\cite{helms}) and Pizzetti formula \cite{pizz},\cite[Part IV, Ch. 3, pp. 287-288]{ch} for polyharmonic functions on the ball in $\mathbb{R}^n$.  Section~\ref{sec-appr} deals with  the best one-sided $L^1$-approximation by harmonic functions.

Let us  remind that a real-valued function $f$ is said to be {\em harmonic} ({\em polyharmonic of degree} $m\ge 2$) in a given domain $\Omega\subset\mathbb{R}^n$ if $f\in C^2(\Omega)$ ($f\in C^{2m}(\Omega)$) and $\Delta f = 0$ ($\Delta ^m  f = 0$) on $\Omega$, where $\Delta$ is the Laplace operator and $\Delta ^m$ is its $m$-th iterate
$$
\Delta f :=\sum_{i=1}^n \frac{\partial ^2 f}{\partial x_i^2},\quad
\Delta ^m  f :=\Delta(\Delta ^{m-1}  f).
$$

For any set $D \subset \mathbb{R}^n$ , denote by $\mathpzc{H}(D)$ ($\mathpzc{H}^m (D), m \ge 2$) the linear space of all functions that are harmonic (polyharmonic of degree $m$) in a domain containing $D$. 
The notation $d\lambda_n$ will stand for the Lebesgue measure in $\mathbb{R}^n$.

\section{Mean-value theorems}
{ \fontfamily{times}\selectfont
 \noindent
\label{sec-mean}                                                  
Let  $B_n (r) := \{\bm{x} \in \mathbb{R}^n ~\vert~
\Vert x\Vert :=\left ( \sum_{i=1}^n x_i^2\right ) ^{1/2}\le r \}$ and 
 $S_n (r) := \{\bm{x} \in \mathbb{R}_n ~\vert~
\Vert x\Vert = r \}$  
be  the  ball and the hypersphere in $\mathbb{R}^n$ with center ${\bm 0}$ and radius $r$. The following famous formulas are basic tools in harmonic function theory and state that for any function $h$ which is harmonic on $B_n(r)$ both the average over $S_n(r)$ and the average over $B_n(r)$ are equal to $h({\bm{0}})$.
\par\medskip\noindent
{\bf The surface mean-value theorem.} {\it If $h \in \mathpzc{H}(B_n (r))$, then
\be\label{Eq-surf}
\frac{1}{\sigma_{n-1} (S_n (r))}\int_{S_n (r)}h\,d\sigma_{n-1}=h({\bm{0}}),
\ee
where $d\sigma_{n-1} $ is the $(n-1)$-dimensional surface measure on the hypersphere $S_n (r)$.}
\par\medskip\noindent
{\bf The volume mean-value theorem.} {\it If $h \in \mathpzc{H}(B_n (r))$, then
\be\label{Eq-volume}
\frac{1}{\lambda_n (B_n (r))}\int_{B_n (r)}h\,d\lambda_{n}=h({\bm{0}}).
\ee}
\par\medskip
The balls are known to be the only sets in $\mathbb{R}^n$ satisfying the surface or the volume mean-value theorem. This means that if $\Omega\subset \mathbb{R}^n$ is a nonvoid domain with a finite Lebesgue measure and if there exists a point $\bm{x}_0\in \Omega$ such that $h(\bm{x}_0 ) = \frac{1}{\lambda_n(\Omega)}\int_\Omega h\,d\lambda_n$ for every function $h$ which is harmonic and integrable on $\Omega$, then $\Omega$ is an open ball centered at $\bm{x}_0$ (see \cite{ghr}). The mean-value properties can also be reformulated in terms of {\em quadrature domains} \cite{sakai}. Recall that $\Omega\subset \mathbb{R}^n$ is said to be a  quadrature domain for $\mathpzc{H}(\Omega)$, if $\Omega$ is a connected open set and there is a Borel measure $d\mu$ with a compact support $K_{\mu}\subset \Omega$ such that
$\int_{\bar{\Omega}}f\,d\lambda_n=\int_{K_\mu}f\,d\mu$  for every $\lambda_n$-integrable harmonic function $f$ on $\Omega$.
Using the concept of quadrature domains, the volume mean-value property is equivalent to the statement that any open ball in $\mathbb{R}^n$ is a quadrature domain and the measure $d\mu$ is the Dirac measure supported at its center.  On the other hand, no domains having "corners" are quadrature domains \cite{gss}. From this point of view, the open hypercube $I_n^\circ$ is not a quadrature domain. Nevertheless, here we prove that the closed hypercube $I_n$ is a quadrature set in an extended sense - there exists a measure $d\mu$ with a compact support $K_\mu$ having the above property with $\Omega$ replaced by $I_n$
 but the condition $K_\mu \subset I_n^\circ$ is violated exactly at the "corners" (Theorem~\ref{Th.1}). This property of $I_n$ is of crucial importance for the best one-sided $L^1$-approximation with respect to $\mathpzc{H}(I_n)$ (Section~\ref{sec-appr}).

Let us denote by $D^{ij}_n$ the $(n-1)$-dimensional hyperplanar segments of $I_n$ defined by
$$
D^{ij}_n:=D^{ij}_n(r):=\{\bm{x} \in  I_n  ~\vert~ \vert x_k\vert\le 
     \vert x_i \vert = \vert x_j \vert, ~k \neq i, j\}, \quad               1 \le i < j \le n.
$$
Denote also
 $$
\omega_k(\bm{x}):=\frac{(r-\hbox{\rm max}\,\{\vert x_1\vert,\vert x_2\vert,\ldots,\vert x_n\vert\})^k}{k!},\quad k\ge 0, 
 $$
 and $d\lambda^{\omega_k}_{m}:=\omega_k\,d\lambda_{m}$.
It can be calculated that 
$$
 \lambda^{\omega_k}_{n}(I_n)=2^{n}n!\frac{r^{n+k}}{(n+k)!},
\quad \lambda^{\omega_k}_{n-1}(P_n)=2^{n}n!\frac{r^{n+k-1}}{(n+k-1)!},  
$$
and
$$
\quad \lambda^{\omega_k}_{n-1}(D_n)=2^{n-1}n!\frac{r^{n+k-1}}{(n+k-1)!}, ~~\hbox{\rm where~~}D_n:=\cup_{1\le i<j\le n} D^{ij}_n.
$$
 \begin{figure}[ht!]
\begin{centering}
\includegraphics[width=0.8\hsize]{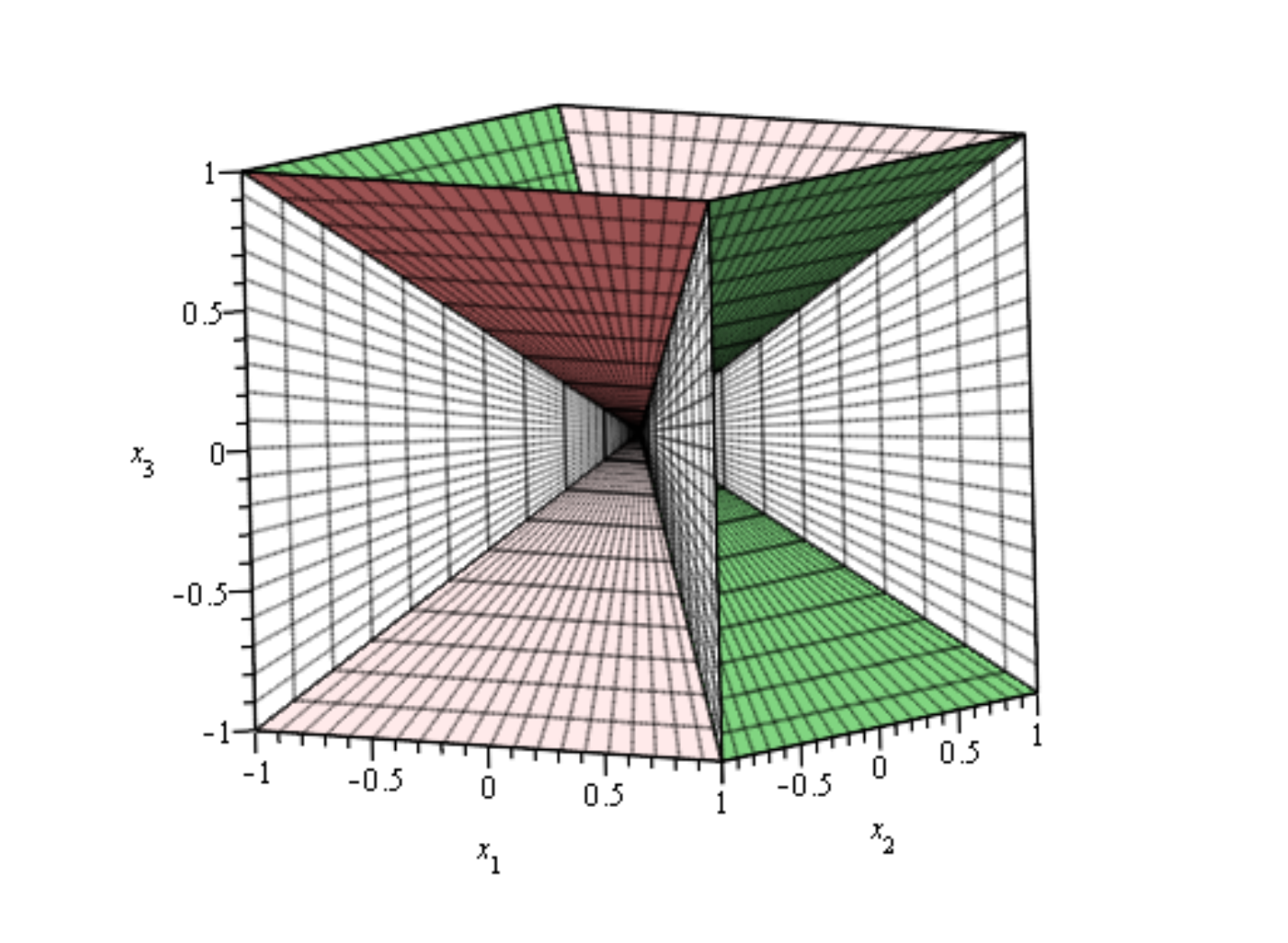}
\caption{\label{fig-mean} The sets $D^{12}_3(1)$ (white), $D^{13}_3(1)$ (green) and $D^{23}_3(1)$ (coral).}
\end{centering}
\end{figure}
The following holds true.
\begin{theorem}\label{Th.1}
If $h \in \mathpzc{H}(I_n)$, then  $h$ satisfies:
\\
\medskip\noindent
(i) {\bf Surface mean-value formula for the hypercube}
\be \label{Eq-surf-cube}
\frac{1}{\lambda_{n-1}(P_n)}\int_{P_n}h\,d\lambda_{n-1}
=\frac{1}{\lambda_{n-1}(D_n)} \int_{D_n}h\,d\lambda_{n-1},
\ee  
(ii) {\bf Volume mean-value formula for the hypercube}
\be \label{Eq-volume-cube}
\frac{1}{\lambda^{\omega_{k}}_{n}(I_n)}\int_{I_n}h\,d\lambda^{\omega_{k}}_{n}=\frac{1}{\lambda^{\omega_{k+1}}_{n-1}(D_n)} \int_{D_n} h\,d\lambda^{\omega_{k+1}}_{n-1},\quad k\ge 0.
\ee  
In particular, both surface and volume mean values of $h$ are attained on $D_n$.
\end{theorem}
\begin{proof}
Set
$$
M_i:=M_i(\bm{x}):=\max_{j\neq i}\vert x_j\vert,
$$
and
$$
\bm{x}^i_{ t}:=(x_1,\ldots,x_{i-1}, t,x_{i+1},\ldots,x_n).
$$
 Using the harmonicity of $h$, we get for $k\ge 1$
 \begin{eqnarray*}
 0 &=& \int_{I_n }\Delta h\,d\lambda_{n}^{\omega_k}=
 \sum_{i=1}^n  \int_{I_n}\omega_k\frac{\partial ^2 h}{\partial x_i^2}
\,d\lambda_{n}\\
&=& -\sum_{i=1}^n  \int_{-r}^r \ldots \int_{-r}^r \frac{\partial \omega_k}{\partial x_i} (\bm{x})\frac{\partial h}{\partial x_i}(\bm{x}) \,dx_idx_1 \ldots dx_{i-1}dx_{i+1}\ldots dx_n\\
&=&- \sum_{i=1}^n  \int_{-r}^r \ldots \int_{-r}^r 
\left \{\left ( \int_{-r}^{-M_i}+
\int_{M_i}^r\right )\hbox{\rm sign}\,x_i \omega_{k-1}(\bm{x})
\frac{\partial h}{\partial x_i}(\bm{x})\,dx_i \right \}\\
&& \times dx_1 \ldots dx_{i-1}dx_{i+1}\ldots dx_n\\
&=&- \sum_{i=1}^n  \int_{-r}^r \ldots \int_{-r}^r \left \{
\int_{M_i}^r \omega_{k-1}(\bm{x}) \frac{\partial }{\partial x_i}  [h(\bm{x}^i_{-x_i})+h(\bm{x}) ]\,dx_i\right \}\\
&& \times dx_1 \ldots dx_{i-1}dx_{i+1}\ldots dx_n.
\end{eqnarray*}
Hence, we have
\begin{eqnarray}\label{Eq.5}
0&=&-\sum_{i=1}^n  \int_{-r}^r \ldots \int_{-r}^r  \{
h(\bm{x}^i_{-r})+h(\bm{x}^i_{+r})\\
&& -[ h(\bm{x}^i_{-M_i})+h(\bm{x}^i_{+M_i})] \}\,dx_1 \ldots dx_{i-1}dx_{i+1}\ldots dx_n \nonumber
\end{eqnarray}
if $k=1$ and 
\begin{eqnarray}\label{Eq.6}
0&=&-\sum_{i=1}^n  \int_{-r}^r \ldots \int_{-r}^r \int_{M_i}^{r}\omega_{k-2}(\bm{x}) [
h(\bm{x}^i_{-x_i})+h(\bm{x})]\,dx_i\nonumber\\
&&\times\,dx_1 \ldots dx_{i-1}dx_{i+1}\ldots dx_n\\
&&+ \sum_{i=1}^n  \int_{-r}^r \ldots \int_{-r}^r
\omega_{k-1}(\bm{x}^i_{+M_i})
 [ h(\bm{x}^i_{-M_i})+h(\bm{x}^i_{+M_i} )]\nonumber \\
&&\times\,dx_1 \ldots dx_{i-1}dx_{i+1}\ldots dx_n \nonumber
\end{eqnarray}
if $k\ge 2$.\\
Clearly, (\ref{Eq.5}) is equivalent to (\ref{Eq-surf-cube}) and from (\ref{Eq.6}) it follows
\be\label{QF-omega}
0=\int_{I_n }\Delta h\,d\lambda_{n}^{\omega_k}=\int_{I_n}h\,d\lambda^{\omega_{k-2}}_{n}
-2 \int_{D_n} h\,d\lambda^{\omega_{k-1}}_{n-1},
\ee
which is equivalent to (\ref{Eq-volume-cube}).
\end{proof}
\par
Let $M:=M(\bm{x}):=\max_{1\le i\le n}\vert x_i\vert$.
Analogously to the proof of Theorem~\ref{Th.1} (ii), Equation (\ref{QF-omega}) is generalized to:
 \begin{corollary} \label{Col.1} 
If $h \in \mathpzc{H}(I_n)$ and $\varphi \in C^2[0,r]$ is such that $\varphi (0)=0$ and $\varphi '(0)=0$, then
\be\label{QF-phi}
0=\int_{I_n }\varphi (r-M)\Delta h\,d\lambda_n=\int_{I_n}\varphi ''(r-M)h\,d\lambda_{n}
-2 \int_{D_n} \varphi ' (r-M) h\,d\lambda_{n-1}.
\ee
\end{corollary}

The volume mean-value formula (\ref{Eq-volume}) was extended by P. Pizzetti  to the following \cite{pizz,ch,boj}.
\par\medskip\noindent
{\bf The Pizzetti formula.} {\it If $g \in \mathpzc{H}^m(B_n(r))$, then
$$\label{Eq-pizz}
\int_{B_n(r)}g\,d\lambda_n=r^n\pi^{n/2}\sum_{k=0}^{m-1}
\frac{r^{2k}}{2^{2k}\Gamma(n/2+k+1)}\frac{\Delta^k g(\bm{0})}{k!}.
$$}
Here we present a similar formula for polyharmonic functions on the hypercube based on integration over the set $D_n$.
\begin{theorem} \label{Th.2}
 If $g\in \mathpzc{H}^m(I_n)$, $m\ge 1$, and $\varphi \in C^{2m}[0,r]$ is such that $\varphi ^{(k)} (0)=0$, $k=0,1,\ldots,2m-1$, then the following identity holds true for any $k\ge 0$:
\be\label{pizz}
\int_{I_n}\varphi ^{(2m)} (r-M) g \,d\lambda_n=2\sum_{s=0}^{m-1}
\int_{D_n}\varphi ^{(2s+1)}(r-M)\Delta^{m-s-1} g\,d\lambda_{n-1},
\ee
where $\varphi ^{(j)}(t)=\frac{d^j \varphi}{dt^j}  (t)$.
\end{theorem}
\begin{proof}
Equation (\ref{pizz}) is a direct consequence from (\ref{QF-phi}): 
\begin{eqnarray*}
 0 &=& \int_{I_n }\varphi (r-M)  \Delta ^m g\,d\lambda_{n}
 \\
 &=& -2
\int_{D^{n}}\varphi ^{(1)} (r-M)\Delta^{m-1} g\,d\lambda_{n-1}+\int_{I_n}\varphi ^{(2)} (r-M)\Delta^{m-1}g\,d\lambda_{n}\\
&=&\ldots =-2\sum_{s=0}^{m-1}
\int_{D_n}\varphi ^{(2s+1)}\Delta^{m-s-1} g\,d\lambda_{n-1}+\int_{I_n}\varphi ^{(2m)} g\,d\lambda_{n}.
\end{eqnarray*}
\end{proof}

\section{A relation to best one-sided $L^1$-approximation by harmonic functions}
{ \fontfamily{times}\selectfont
 \noindent
\label{sec-appr}
Theorem~\ref{Th.1} suggests that for a certain positive cone  in $C(I_n )$ the set $D_n$ is a characteristic set for the best one-sided $L^1$-approximation with respect to $\mathpzc{H}(I_n)$ as it is explained and illustrated by the examples presented below. 

For a given $f\in C(I_n )$, let us introduce the following subset of 
$\mathpzc{H}(I_n)$:
$$
\mathpzc{H}_{-}(I_n,f):=\{h\in \mathpzc{H}(I_n)~|~h\le f \hbox{\rm ~on~} I_n\}.
$$
A harmonic function $h^f_*\in \mathpzc{H}_{-}(I_n,f)$ is said to be {\em a best one-sided $L^1$-approximant from below to $f$ with respect to $\mathpzc{H}(I_n)$} if
$$
\Vert f- h^f_*\Vert_1\le \Vert f- h\Vert_1\hbox{~~for every~~} h\in \mathpzc{H}_{-}(I_n,f),
$$                       
where
$$
\Vert g\Vert_1:=\int_{I_n}\vert g\vert \,d\lambda_{n}.
$$
Theorem~\ref{Th.1} (ii) readily implies the following (\cite{{ag,gss}}).
\begin{theorem} \label{Th.3} 
Let $f\in C(I_n)$ and $h^f_*\in \mathpzc{H}_{-}(I_n,f)$. Assume further that the set $D_n$ belongs to the zero set of the function $f-h^f_*$. Then $h^f_*$ is a best one-sided $L^1$-approximant from below to $f$ with respect to $\mathpzc{H}(I_n)$.
\end{theorem}
\begin{corollary}\label{Col.2}  
If $f\in C^1(I_n)$, any solution $h$ of the problem 
\be\label{Eq-interp}
h_{\,|D_n} = f_{\,|D_n},~~\nabla h_{\,|D_n} = \nabla f_{\,|D_n},~~h\in \mathpzc{H_-}(I_n,f),
\ee
is a best one-sided $L^1$-approximant from below to $f$ with respect to $\mathpzc{H}(I_n)$.
\end{corollary}
\begin{corollary}\label{Col.3} 
If  $f(\bm{x})=g(\bm{x})\prod_{1\le i<j\le n}(x_i^2-x_j^2)^2$, where $g\in C(I_n)$ and $g\ge 0$ on $I_n$, then $h^f_*(\bm{x})\equiv 0$ is a best one-sided $L^1$-approximant from below to $f$ with respect to $\mathpzc{H}(I_n)$.
\end{corollary}
\begin{example}{\rm
Let $n=2$, $r=1$ and $f_1(x_1,x_2) =  x_1^2x_2^2$. By Corollary~\ref{Col.2},  the solution $h^{f_1}_*(x_1,x_2)=-x_1^4/4+\frac{3}{2}x_1^2x_2^2-x_2^4/4$ of the interpolation problem (\ref{Eq-interp}) with $f=f_1$
is  a best one-sided $L^1$-approximant from below to $f_1$ with respect to $\mathpzc{H}(I_2)$ and $\Vert f_1-h^{f_1}_*\Vert_1 =8/45$. 
Since the function $f_1$ belongs to the positive cone of the partial differential operator $\mathcal{D}^4_{2,2}:=\frac{\partial ^4}{\partial x_1^2\partial x_2^2}$ (that is, $\mathcal{D}^4_{2,2}f_1>0$), one can compare the best harmonic one-sided $L^1$-approximation to $f_1$ with the corresponding approximation from the linear subspace of $C(I_2)$:
\begin{eqnarray*}
\mathpzc{B}^{2,2}(I_2):=\{b\in C(I_2)~\vert~b(x_1,x_2)=\sum_{j=0}^1[a_{0j}(x_{1})x_2^j+a_{1j}(x_{2})x_1^j]\}.
\end{eqnarray*}
 The  possibility for explicit constructions  of best one-sided $L^1$-approximants from  $\mathpzc{B}^{2,2}(I_2)$, is studied in \cite{dp-jat}.  The functions $f_1-b^{f_1}_*$ and $f_1-b_{f_1}^*$, where $b^{f_1}_*$ and $b_{f_1}^*$ are the unique best one-sided $L^1$-approximants  to $f_1$ with respect to $\mathpzc{B}^{2,2}(I_2)$ from below and above,  respectively, play the role of basic error functions of the canonical one-sided $L^1$-approximation by elements of $\mathpzc{B}^{2,2}(I_2)$. 
For instance, $b^{f_1}_*$ can be constructed as the unique interpolant to $f_1$ on the boundary $\Diamond:=\{(x_1,x_2)\in I_2~|~\vert x_1\vert +\vert x_2\vert=1\}$ of the inscribed square and $\Vert f_1-b^{f_1}_*\Vert_1=14/45$  (Fig.~\ref{fig-f1}).
\begin{figure}[ht!]
\begin{centering}
\includegraphics[width=0.48\hsize]{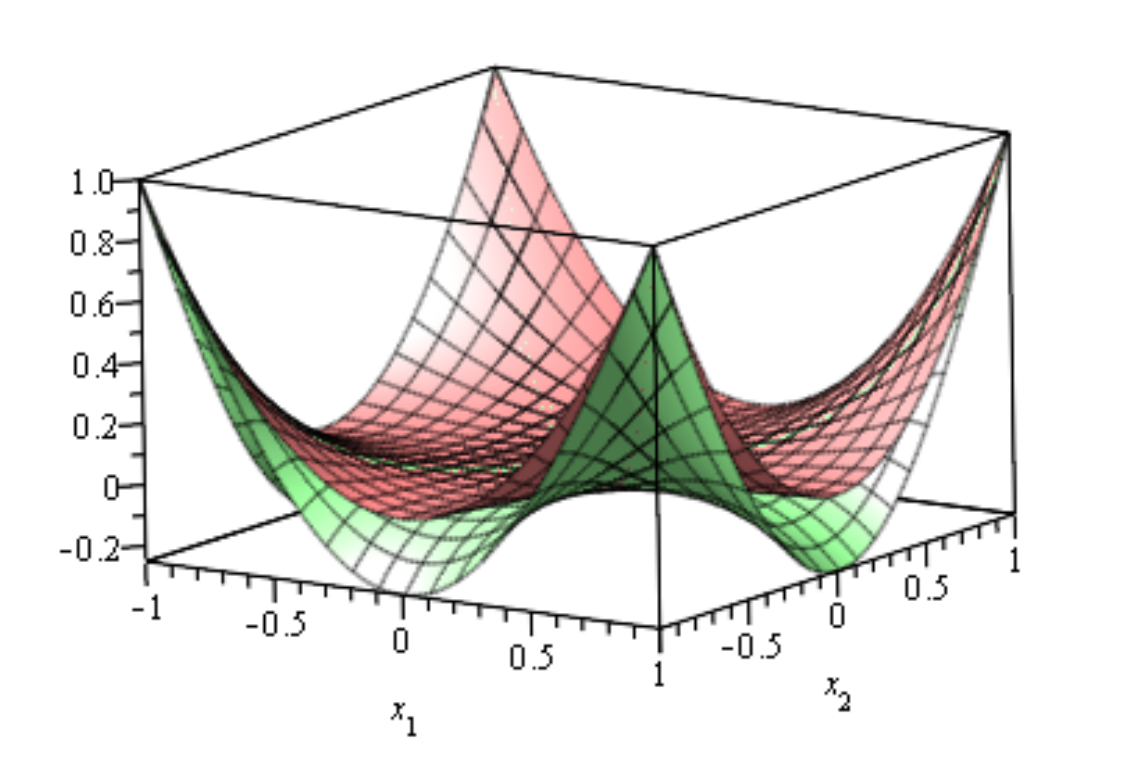}\hfill
\includegraphics[width=0.48\hsize]{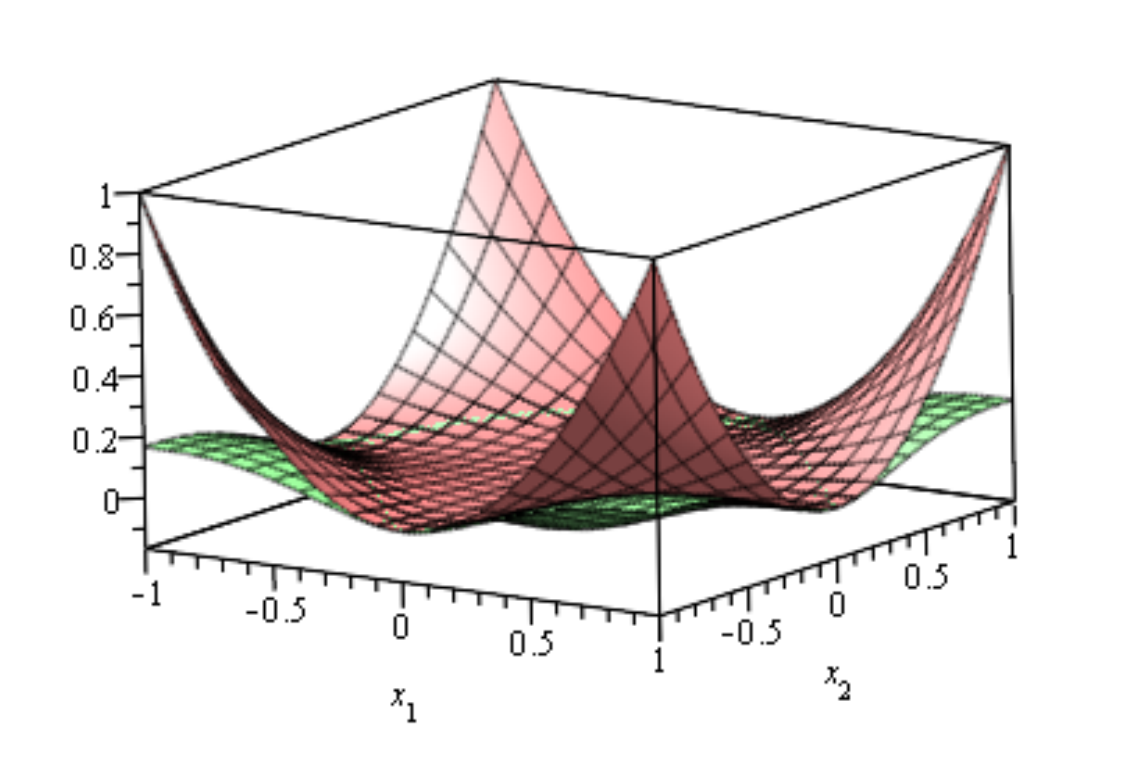}
\caption{\label{fig-f1} The graphs of the function $f_1(x_1,x_2)=  x_1^2x_2^2$ (coral) and its best one-sided $L^1$-approximants from below, $h^{f_1}_*$ with respect to $ \mathpzc{H}(I_2)$ (left) and $b^{f_1}_*$ with respect to $ \mathpzc{B}^{2,2}(I_2)$ (right).}
\end{centering}
\end{figure}}
\end{example}
\begin{example}{\rm
Let $n=2$, $r=1$ and $f_2(x_1,x_2) = x_1^8+14x_1^4x_2^4+x_2^8 $. 
The solution  
$h^{f_2}_*(x_1,x_2)=x_1^8+x_2^8-28(x_1^6x_2^2+x_1^2x_2^6)+70x_1^4x_2^4$ of 
 (\ref{Eq-interp}) with $f=f_2$ is  a best one-sided $L^1$-approximant from below to $f_2$ with respect to $\mathpzc{H}(I_2)$ and $\Vert f_2-h^{f_2}_*\Vert =8/75$. It can also be verified that $\Vert f_2-b^{f_2}_*\Vert =121/900$ (see Fig.~\ref{fig-f2}).
\begin{figure}[ht!]
\begin{centering}
\includegraphics[width=0.48\hsize]{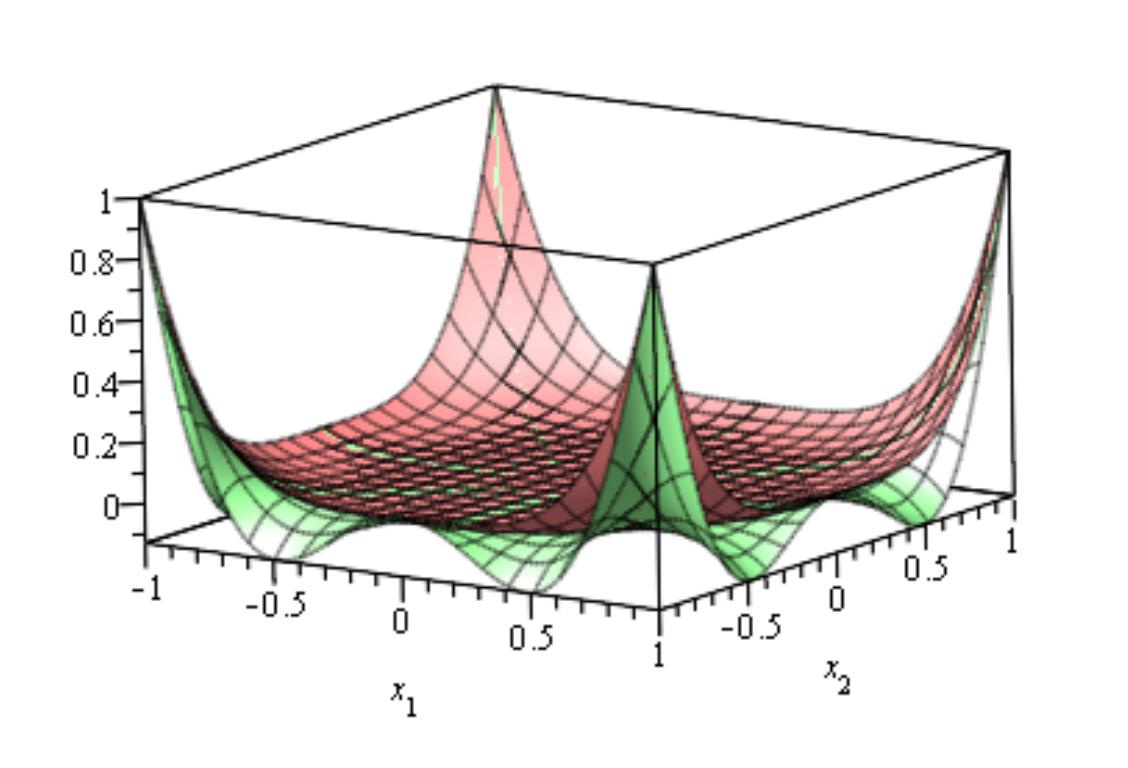}\hfill
\includegraphics[width=0.48\hsize]{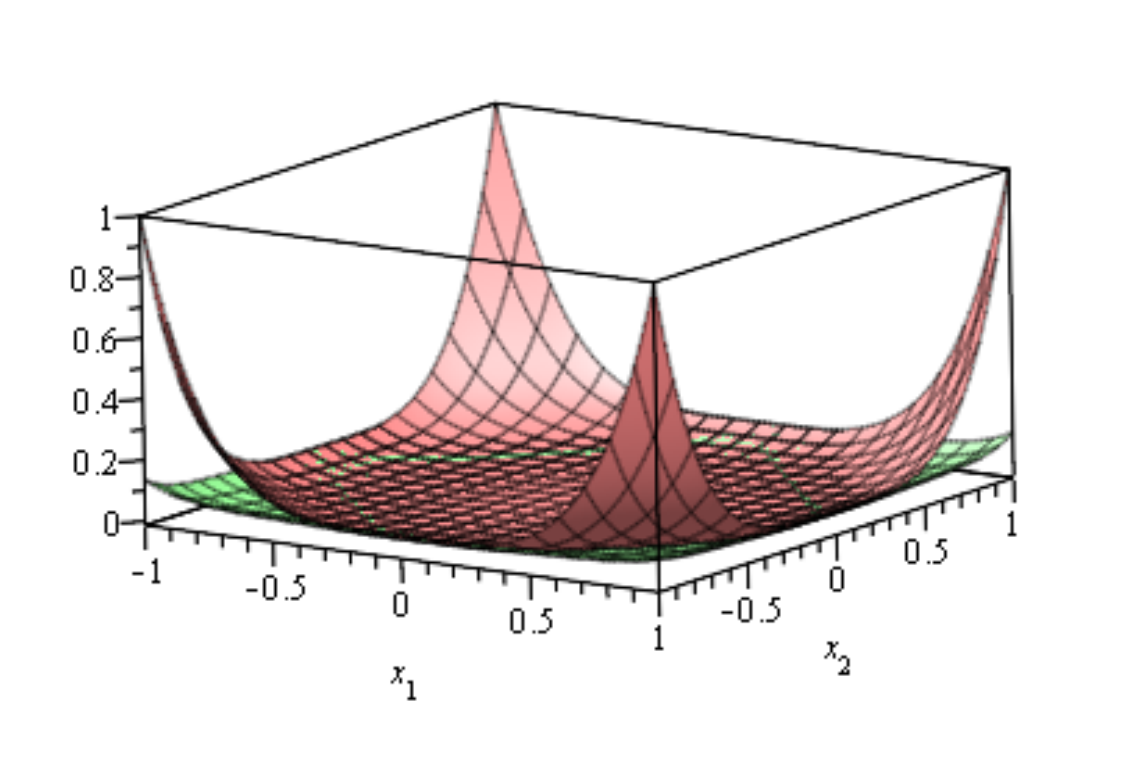}
\caption{\label{fig-f2} The graphs of the function $f_2(x_1,x_2) = x_1^8+14x_1^4x_2^4+x_2^8$ (coral) and its best one-sided $L^1$-approximants from below, $h^{f_2}_*$ with respect to $\mathpzc{H}(I_2)$ (left) and $b^{f_2}_*$ with respect to $ \mathpzc{B}^{2,2}(I_2)$  (right).}
\end{centering}
\end{figure}}
\end{example}
\begin{remark} {\rm
Let $\varphi \in C^2[0,r]$ is such that $\varphi (0)=0$, $\varphi '(0)=0$, and $\varphi '\ge 0$, $\varphi ''\ge 0$ on $[0,r]$.
It follows from (\ref{QF-phi}) that Theorem \ref{Th.3} also holds for the best weighted $L^1$-approximation from below with respect to $\mathpzc{H}(I_n)$ with weight $\varphi '' (r-M)$. The smoothness requirements were used for brevity and wherever possible they can be weakened in a natural way. }
\end{remark}

%

\end{document}